\documentclass[a4paper,11pt,leqno]{amsart}
\usepackage[applemac]{inputenc}
\usepackage{epsfig,graphics}
\usepackage{amsthm,amssymb,amsbsy,bbm,a4wide,verbatim}
\usepackage{longtable}
\usepackage{rotating}

\def\Gal{\mathrm{Gal}}
\def\kk{\mathbbm{k}}
\def\Z{\mathbbm{Z}}

\def\F{\mathbbm{F}}

\def\Q{\mathbbm{Q}}

\def\Ker{\mathrm{Ker}}
\def\Aut{\mathrm{Aut}}
\def\Sp{\mathrm{Sp}}

\def\Res{\mathrm{Res}}

\def\la{\lambda}

\def\eps{\varepsilon}

\def\onto{\twoheadrightarrow}

\def\GL{\mathrm{GL}}
\def\PGL{\mathrm{PGL}}
\def\SL{\mathrm{SL}}
\def\PSL{\mathrm{SL}}
\def\GU{\mathrm{GU}}

\def\SU{\mathrm{SU}}
\def\PSU{\mathrm{PSU}}
\def\PSL{\mathrm{PSL}}

\newcommand{\pPGL}{\mathfrak{q}}

\newcommand{\pr}{a}

\newcommand{\sym}{\mathfrak{S}}

\newtheorem{theor}{Theorem}
\newtheorem{lemma}[theor]{Lemma}
\newtheorem{prop}[theor]{Proposition}
\newtheorem{cor}[theor]{Corollary}

\newtheorem{remark}[theor]{Remark}
\numberwithin{theor}{section}

\title{Image of the braid groups inside the finite Temperley-Lieb algebras}
\author{Olivier Brunat}
\author{Ivan Marin}
\address{Institut de MathÈmatiques de Jussieu \\ Universit\'e Denis
Diderot - Paris 7 \\ Case 7012 \\ 75205 Paris Cedex 13 \\ France }
\email{brunat@math.jussieu.fr}
\address{LAMFA \\ Universit\'e de Picardie-Jules Verne \\ 33 rue Saint-Leu \\ 80039 Amiens Cedex 1 \\ France }
\email{ivan.marin@u-picardie.fr} 

\date{November 14, 2013.}
\begin{document}

\begin{abstract}
We determine the image of the braid groups inside the Temperley-Lieb
algebras, defined over finite field, in the semisimple case, and for
suitably large (but controlable) order of the defining (quantum) parameter. 
We also prove that, under natural conditions on this
parameter, the representations of the Hecke algebras over a finite
field are unitary for the action of the braid groups.
\end{abstract}

\maketitle



\section{Introduction}

Let $B_n$ denote the braid group on $n$ strands. A natural question
concerns the image of $B_n$ inside its classical linear representations,
the most classical ones being the ones which factor through the Hecke
algebra $H_n(\alpha)$ of type $A_{n-1}$, such as the Burau or the Jones representation.
Inside an infinite field, the determination of the Zariski closure
of such representations in the generic case is completely known by \cite{FLW} and \cite{LIETRANSP} ;
actually the more general cases of the representations of the Birman-Wenzl-Murakami
algebra and of the Hecke algebras for other reflection groups is also known by \cite{BMW,IH2}, and more precise information on the
Jones representation can be found in \cite{FLW} and \cite{KUPERBERG} in the non-generic case.

In vague terms, the theory of `strong approximation' should imply that, for `almost all'
maximal ideals $\mathfrak{m}$ of the ring of definition $\Z[\alpha,\alpha^{-1}]$ of the representation, the image of $B_n$
should be the set of points over $\F_q = \Z[\alpha,\alpha^{-1}]/\mathfrak{m}$ of the corresponding algebraic group.
However, it is unclear to us, partly because $\Z[\alpha,\alpha^{-1}]$ has Krull dimension 2,
whether strong approximation techniques can lead to reasonably precise results in this case.

In the case of the Jones representation, the problem is equivalent to determining the image of
the braid group inside the Temperley-Lieb algebra, defined over a finite field. It is a natural
generalization of a problem which has already been studied, in the case of the Burau
representation, in the realm of inverse Galois theory.


Indeed, by \cite{SZ2} and \cite{WAGNER} (see also \cite{MALLE} II \S 2, Theorem~2.3), we have
the following result.

\begin{theor} (Sere\v{z}kin-Zalesski{\u\i}, Wagner) \label{theowagner} Let $n \geq 3$, $q = p^m$ and $H$ a primitive
subgroup of $\GL_n(\F_q)$ generated by semisimple pseudo-reflections of
order at least $3$, then one of the following holds.
\begin{enumerate}
\item $\SL_n(\F_{\tilde{q}}) \leq H \leq \GL_n(\F_{\tilde{q}})$ for $\tilde{q} | q$ or
\item $\SU_n(\tilde{q}) \leq H \leq \GU_n(\tilde{q})$ for $\tilde{q}^2 | q$ 
\end{enumerate}
or $n \leq 4$ and $H \simeq \GU_n(2)$. In the latter case the pseudo-reflections have order $3$.
\end{theor}

As noticed in \cite{VOLK}, this theorem determines the image of the
braid group inside the Burau representation.
A natural question, raised by Strambach and V\"olklein in \cite{VOLK}, is to determine the image of the braid group,
and of other generalized braid groups, inside the representations of the Hecke algebra representations over
a finite field. This is relevant to the question of determining rigid geometric Galois actions. The special
case of the Temperley-Lieb algebra can thus also be seen as a first step towards answering this question.

We thus consider the Hecke algebra $H_n(\alpha)$ as defined over a finite field $\F_q$,
meaning that we choose $\alpha \in \F_q^{\times}$ and consider the quotient of the
group algebra $\F_q B_n$ by the relations $(\sigma_i +1)(\sigma_i - \alpha) = 0$,
where the $\sigma_i$ are the usual Artin generators of $B_n$. In studying the
representation theory of $H_n(\alpha)$,
an important integer $e$ is the smallest positive one such that
$[e]_{\alpha} = 1+\alpha+\dots+\alpha^{e-1} =0$. If $\alpha=1$, then $e = p = char. \F_q$, otherwise $[e]_{\alpha} = (\alpha^e-1)/(\alpha-1) = 0$
means $\alpha^e = 1$. 

The irreducible representations of $H_n(\alpha)$ are in 1-1 correspondence
with the partitions $\la = (\la_1 \geq \la_2 \geq \dots)$ of $n$ which are $e$-restricted, meaning that $\la_i - \la_{i+1} < e$
for all $i$, and $H_n(\alpha)$ is (split) semisimple if and only if all partitions of $n$ are $e$-restricted,
meaning $e > n$ (see \cite{MATHAS}, cor. 3.44). 
In that case, all the irreducible 
representations are afforded by the classical Specht modules,
and are thus reductions modulo $\mathfrak{m}$ of
the representations in characteristic $0$.


We focus on the partitions of at most two rows, and denote $c(n,r)$
the dimension of the representation associated to $[n-r,r]$.
These representations are exactly the irreducible representations which
factors through the Temperley-Lieb algebra. This algebra
$TL_{n}(\alpha)$ can be defined as the quotient of the Hecke algebra
$H_{n}(\alpha)$ by the relation $\sigma_2 \sigma_1 \sigma_2 + 
\sigma_1 \sigma_2 +  \sigma_2 \sigma_1 +  \sigma_1 +  \sigma_2 + 1 = 0$.

We now can state the main result of this paper.

\begin{theor} 
Let $n$ be a positive integer and let $p$ be a prime number. 
Let $\alpha \in \overline{\F_p}^{\times}$ of order $e > n $
and $e \not\in \{ 2,3,4,5,6,10 \}$,
and let $\F_q = \F_p(\alpha)$.
\begin{enumerate}
\item Let $\la \vdash n $ be a partition with at most two rows. 
If $R : B_n \to \GL_N(\F_q)$ denotes the representation of
$H_{n}(\alpha)$ associated to $\la$, then
\begin{itemize}
\item either $\F_p(\alpha + \alpha^{-1}) = \F_p(\alpha) = \F_q$
and 
$R(B_n)$ contains $\SL_N(\F_q)$,
\item or $\F_p(\alpha + \alpha^{-1}) = \F_{q^{1/2}}$ and,
up to conjugacy,
$ \SU_N(q^{1/2}) \subset R(B_n) \subset \GU_N(q^{1/2})$.
\end{itemize}
\item 
Let $G$ be the image of $B_n$ in $TL_{n}(\alpha)^{\times} = \prod_{r=0}^{n} \GL_{c(n,r)}(\F_p)$.
Then 
\begin{itemize}
\item either $\F_p(\alpha + \alpha^{-1}) = \F_p(\alpha) = \F_q$
and 
$G$ contains $\prod_{r=0}^{n} \SL_{c(n,r)}(\F_q)$,
\item or $\F_p(\alpha + \alpha^{-1}) = \F_{q^{1/2}}$ and
$$
\prod_{r=0}^{n} \SU_{c(n,r)}(q^{1/2})  \subset P R(B_n)P^{-1} \subset 
\prod_{r=0}^{n} \GU_{c(n,r)}(q^{1/2})
$$
for some $P \in \prod_{r=0}^{n} \GL_{c(n,r)}(\F_q)$.
\end{itemize}
\end{enumerate}
\label{main}
\end{theor}

The article is organized as follows. In Section~\ref{preliminaire}, we
give some results that we need in the sequel. 
Then in Section~\ref{red}, we reduce the proof of
Theorem~\ref{main} to the following statements.


\begin{theor} \label{theorinductGLGL} Let $p$ be a prime and $q$ a
$p$-power. Let $\Gamma < \GL_N(\F_q)$ with $N
\geq 5$ and $q > 3$, 
such that
\begin{enumerate}
\item $\Gamma$ is absolutely irreducible.
\item $\Gamma$ contains 
$\SL_{\pr}(\F_q)$
with $\pr\geq N/2$.
\end{enumerate}
If $N\neq 2a$, then $\Gamma$ contains $\SL_N(\F_q)$. Otherwise, either
$\Gamma$ contains $\SL_N(\F_q)$, or $\Gamma$ is a subgroup of
$\GL_{N/2}(\F_q)\wr\sym_2$.
\end{theor}
We also need the unitary version of this result.

\begin{theor} \label{theorinductGUGU} Let $p$ be a prime and $q$ a
$p$-power. Let $\Gamma < \GU_N(\F_q)$ with $N
\geq 5$ and $q > 3$, such that
\begin{enumerate}
\item $\Gamma$ is absolutely irreducible.
\item $\Gamma$ contains 
$\SU_{\pr}(q)$
with $\pr\geq N/2$.
\end{enumerate}
If $N\neq 2a$, then $\Gamma$ contains $\SU_N(q)$. Otherwise, either
$\Gamma$ contains $\SU_N(q)$, or $\Gamma$ is a subgroup of
$\GU_{N/2}(q)\wr\sym_2$.
\end{theor}

\begin{remark} 
\end{remark}
The assumptions used in Theorem~\ref{theorinductGLGL} are clearly too
strong. However, this result does not hold in general
(that is : arbitrary field and arbitrary $N$), as exemplified
by $N = 4$, $q = 2$. In that case, $\operatorname{Sp}_4(\F_2)$ is an
absolutely irreducible subgroup of order $720$ of $\GL_4(\F_2)$
containing $\SL_2(\F_2)\times\SL_2(\F_2)$.
It is clearly not contained in $\GL_2(\F_2) \wr \mathfrak{S}_2$,
which has order $72$.

\medskip

{\bf Acknowledgements.} I.M. thanks R. Rouquier for pointing out the
paper \cite{VOLK} to him. We thank M. Cabanes for fruitful discussions,
and K. Magaard for indicating the
reference~\cite{kantor} to us, which led to a dramatic improvement of a
previous version of this paper. We also thank the referee for a careful
reading of our manuscript.

\section{Preliminary results}\label{preliminaire}

For the convenience of the reader, 
we give a proof of the following well-known result.
\begin{lemma}Let $\kk$ be a field and a positive integer $n$.
Write $\pPGL:\GL_n(\kk)\rightarrow\PGL_n(\kk)$ for the natural projection.
If $G\leq \GL_n(\kk)$ is such that $\PSL_n(\kk)\leq \pPGL(G)$, then
$\SL_n(\kk)\leq G$. 
\label{lem:quotientsln}
\end{lemma}

\begin{proof}
First, assume that $\SL_n(\kk)$ is perfect, i.e.
$\SL_n(\kk)\notin\{\SL_2(\F_2),\SL_2(\F_3)\}$.  Set $G'=[G,G]$. In
particular, $G'\leq \SL_{n}(\kk)$, and $\pPGL(G')=\PSL_{n}(\kk)$. Thus,
$\SL_{n}(\kk)=\mu_n(\kk)G'$, where $\mu_n(\kk)$
is the group of the $n$-th roots
of $1$ in $\kk$. Now, we have
$$\SL_n(\kk)=[\SL_n(\kk),\SL_n(\kk)]=[\mu_n(\kk)G',\mu_n(\kk)G']=[G',G']\leq
G,$$
as required.
Moreover, we easy check that the statement holds
for $\SL_2(\F_2)$ and $\SL_2(\F_3)$.
\end{proof}

Now, we recall Goursat's lemma (sometimes also attributed to P. Hall), which describes the subgroups of a direct
product, and that we need in the following.
\begin{lemma}(Goursat's lemma)
Let $G_1$ and $G_2$ be two groups, $H\leq G_1\times G_2$, and denote by
$\pi_i:H\rightarrow G_i$ the natural projections. 
Write $H_i=\pi_i(H)$ and $H^i=\ker(\pi_{i'})$,
where $\{ i, i' \} = \{ 1, 2 \}$.
Then there is an isomorphism  $\varphi:H_1/H^1\rightarrow H_2/H^2$ 
such that 
\begin{equation}
\label{eq:goursat}
H=\{(h_1,h_2)\in H_1\times H_2
\ |\ \varphi(h_1H^1)=h_2H^2\}.
\end{equation}
\label{lem:goursat}
\end{lemma}

\begin{lemma}Let $p$ be a prime and $q>3$ be a $p$-power. Let
$a_1,\ldots,a_n$ be distinct positive integers, and 
$H\leq\GL_{a_1}(\F_q)\times\cdots\times\GL_{a_n}(\F_q)$.  For $1\leq
i\leq n$, write $p_i:H\rightarrow \GL_{a_i}(\F_{q})$ for the natural
projection. If $\SL_{a_i}(\F_q)\leq p_i(H)$,
then $H$ contains $\SL_{a_1}(\F_q)\times\cdots\times\SL_{a_n}(\F_q)$.
\label{lem:levi}
\end{lemma}

\begin{proof}
We prove the statement by induction on $n$. The case $n=1$ is trivial so
we can assume $n\geq 2$. Set
$G_1=\GL_{a_1}(\F_q)\times\cdots\GL_{a_{n-1}}(\F_q)$ and
$G_2=\GL_{a_r}(\F_q)$. As in Lemma~\ref{lem:goursat}, for $1\leq i\leq
2$, write $\pi_i:H\rightarrow G_i$ for the natural projections, 
$H_i=\pi_i(H)$ and $H^i=\ker(\pi_{i'}(H))$.
Hence, by Lemma~\ref{lem:goursat}, there is an isomorphism
$\varphi:H_1/H^1\rightarrow H_2/H^2$ such that Equation~(\ref{eq:goursat})
holds.
We have $\SL_{a_n}(\F_q)\leq H_2$ by assumption. Moreover, by the
induction assumption, $H_1$ contains
$\SL_{a_1}(\F_q)\times\cdots\times\SL_{a_{n-1}}(\F_q)$.  Since
$[H_1,H_1]=\SL_{a_1}(\F_q)\times\cdots\times\SL_{a_{n-1}}(\F_q)$, it
follows that the non-cyclic decomposition factors of $H_1$ are
$\PSL_{a_1}(\F_q),\ldots,\PSL_{a_{n-1}}(\F_q)$. Similarly, we get that the
non-cyclic decomposition factor of $H_2$ is $\PSL_{a_n}(\F_q)$.
Therefore, using $\varphi$, we deduce that $H_2/H^2$ is solvable.  Its
commutator subgroup is a quotient of $[H_2,H_2]=\SL_{a_n}(\F_q)$. Since
it is solvable, it has to be trivial. In particular, $H_2/H^2$ is
abelian, implying that $H_1/H^1$ is abelian. Thus, $[H_1,H_1]\leq H^1$
and $[H_2,H_2]\leq H^2$. The result follows.
\end{proof}

We recall the following classical fact.

\begin{lemma} \label{lemrepunit} Let $G$ be a group, and $\rho : G \to \GL_N(\F_{q^2})$
be an absolutely irreducible representation such that $\eps \circ \rho^*$
is isomorphic to $\rho$, where $\eps \in \Aut(\F_{q^2})$ has
order 2.
Then there exists $P \in \GL_N(\kk)$ such
that $P \rho(g)P^{-1} \in \GU_N(q)$ for all $g \in G$.
\end{lemma}
\begin{proof} By assumption there exists $P \in \GL_N(\F_{q^2})$ such
that $\eps( ^t \rho(g^{-1})) = P \rho(g) P^{-1}$ for all $g \in G$,
that is $\eps( ^t \rho(g^{-1})) P= P \rho(g) $. It follows that 
$
\rho(g) = \eps( ^t (\eps( ^t \rho(g)^{-1}))^{-1}) = 
\eps( ^t P^{-1} ) \eps( ^t \rho(g)^{-1} ) \eps( ^t P)
= 
\eps( ^t P^{-1} ) P \rho(g) P^{-1} \eps( ^t P)$ hence $^t \eps(P)^{-1} P$
intertwines $\rho$ with itself. Since $\rho$ is absolutely irreducible,
this implies $^t \eps(P)^{-1} P  = \mu \in \F_{q^2}^{\times}$.  Moreover
$\mu = {}^t \mu = {}^t P \eps(P)^{-1}$, hence $\mu = \eps(P)^{-1}\,  ^t P$
and $\mu \eps(\mu) = \eps(P)^{-1}\, ^t P ^t P^{-1} \eps(P) = 1$. It
follows that $1 \mapsto 1$, $\eps \mapsto \mu$ defines a 1-cocycle of
$\Gal(\F_{q^2} | \F_q) \simeq \Z/2\Z$ with values in
$\F_{q^2}^{\times}$. By Hilbert's Theorem 90 such a 1-cocycle is a
coboundary, that is there exists $\la \in \F_{q^2}^{\times}$ such that
$\mu = \eps(\la) \la^{-1}$. Hence $^t \eps(P)^{-1} P = \eps(\la)
\la^{-1}$, and $ ^t \eps(\la P) = \la P$. Up to replacing $P$ by $\la
P$, we can thus assume that $^t \eps(P) = P$. Then $ ^t(\eps(\rho(g) X))
P (\rho(g)Y) ^t\eps(X) PY$ for all $X,Y \in \F_{q^2}^{\times}$, that is
to say that all $\rho(g)$ preserves a non-degenerate hermitian form over
$\F_{q^2}^N$, with respect to $\eps$. Since all such form are equivalent
(see e.g. 
\cite{grove}, Theorem 10.3) this implies the conclusion.\end{proof}

\section{Proof of Theorem~\ref{main}}\label{red}

We use the notations of Theorem~\ref{main}. In all what follows, we
assume that Theorems~\ref{theorinductGLGL} and~\ref{theorinductGUGU}
hold.
In \S\ref{par31}, \S\ref{par32} and \S\ref{par33}, we assume
$\F_p(\alpha + \alpha^{-1}) = \F_p(\alpha) = \F_q$. Then we will indicate
in \S\ref{par34} the modifications that are needed for the unitary case.

\subsection{Technical preliminaries}
\label{par31}

We first prove the following proposition. 
\begin{prop} \label{propR1R2SL} Assume $n \geq 5$, $r_1,r_2 \leq n/2$ with $r_1 \neq r_2$, and let $N_i = c(n,r_i)$, $R_i : B_n \to \GL_{N_i}(\F_q)$
denote the representation associated to $[n-r_i,r_i]$. Assume that $\F_p(\alpha) = \F_p(\alpha + \alpha^{-1})$, and let $R = R_1 \oplus
R_2$. If $R_i(B_n) \supset \SL_{N_i}(\F_q)$ for $i \in \{1,2 \}$, then $R(B_{n}) \supset \SL_{N_1}(\F_q) \times \SL_{N_2}(\F_q)$. 
\end{prop}
\begin{cor} \label{corR1R2SL} Assume $n \geq 6$, $r \leq n/2$, $N = c(n,r)$, and $R : B_n \to \GL_N(\F_q)$
the representation associated to $[n-r,r]$. Assume that $\F_p(\alpha) = \F_p(\alpha + \alpha^{-1})$.
Assume that $[n-r-1,r]$ and $[n-r,r-1]$ are
partitions of $n-1$, with associated representations
$R_i : B_{n-1} \to \GL_{N_i}(\F_q)$ and that the restriction of
$R$ to $B_{n-1}$ is $R_1 \oplus R_2$. Finally, assume that $R_i(B_{n-1}) \supset \SL_{N_i}(\F_q)$
for $i =1,2$. Then $R(B_{n-1}) \supset \SL_{N_1}(\F_q) \times \SL_{N_2}(\F_q)$.
\end{cor}

We will need the following lemmas.

When $[n-r,r]$ is a partition of $n$ with associated representation $R : B_n \to \GL_{c(n,r)}(\F_q)$,
we let $a(n,r) = \dim \Ker(R(\sigma_1) + 1)$ and $b(n,r) = \dim \Ker(R(\sigma_1) - \alpha)$. Clearly
$a(n,r) + b(n,r) = c(n,r)$. When $[n-r,r]$ is not a partition of $n$ we let $a(n,r) = b(n,r) = c(n,r)= 0$.

\begin{lemma} \label{lem:multiTL} If $[n-r,r]$ is a partition of $n$ and $n \geq 5$, then $a(n,r) > b(n,r)$.
\end{lemma}
\begin{proof}
The statement holds true for $n = 5$ by a direct computation : $ a(5,0) = 1, b(5,0) = 0$,
$ a(5,1) = 3, b(5,1) = 1$, $ a(5,2) = 3, b(5,2) = 2$ (note that it is not true for $n=4$,
as $a(4,2) = b(4,2) = 1$).

We prove it
by induction on $n$, now assuming $n \geq 6$. By the branching
rule we have $a(n,r) = a(n-1,r) + a(n-1,r-1)$
and $b(n,r) =b(n-1,r) + b(n-1,r-1)$. Since at least one
of the two couples $[n-1-r,r]$ and $[n-r,r-1]$ is a partition of $r$,
the induction assumption immediately implies the conclusion.

\end{proof}

\begin{lemma} \label{lem:gprime} Let $G$ be a group, $\kk$ a field and $R_1,R_2 : G \to \GL_N(\kk)$
with $N \geq 2$ two representations, such that $(R_1)_{|G'} =
(R_2)_{|G'}$, where $G'$ denotes the commutator subgroup of $G$,
and such that $R_1(G') = R_2(G') \supset \SL_N(\kk)$. Then
there exists a character $\eta : G \to \kk^{\times}$ such
that $R_2 = R_1 \otimes \eta$.
\end{lemma}
\begin{proof} Let $\eta : G \to \GL_N(\kk)$ the
map defined by $\eta(g) = R_2(g) R_1(g)^{-1}$. 
For $g \in G$ and $h \in G'$,
we have $\eta(gh) = R_2(g) (R_2(h) R_1(h)^{-1}) R_1(g)^{-1}
=  R_2(g) R_1(g)^{-1} = \eta(g)$,
and also $\eta(gh) = \eta(ghg^{-1}.g) 
= R_2(ghg^{-1}) R_2(g) R_1(g)^{-1} R_1(ghg^{-1})
= R_2(ghg^{-1}) \eta(g) R_2(ghg^{-1})$.
It follows that $\eta(g)$
centralizes $R_2(g G' g^{-1}) = R_2(G') \supset \SL_N(\kk)$
hence $\eta(g) \in \kk^{\times}$. Then
$\eta(g_1 g_2) = R_2(g_1) (R_2(g_2) R_1(g_2)^{-1})R_1(g_1)^{-1}   
= R_2(g_1) \eta(g_2) R_1(g_1)^{-1}    
= R_2(g_1)  R_1(g_1)^{-1}    \eta(g_2)
= \eta(g_1)    \eta(g_2)$
for all $g_1,g_2 \in G$, which proves the claim.

\end{proof}

We can now prove Proposition~\ref{propR1R2SL}.
\begin{proof}
First note that the assumption $R_i(B_{n}) \supset \SL_{N_i}(\F_q)$
imply $\F_q = \F_p(\alpha) = \F_p(\alpha + \alpha^{-1})$.
 We use the notations of Goursat's lemma : $H = R(B_{n})$ is the
subgroup of $\GL_{N_1}(\F_q) \times \GL_{N_2}(\F_q)$ defined by
$ H = \{ (x,y) \in H_1 \times H_2 \ | \ \varphi(x H^1) = y H^2 \}$,
where $R_i(B_{n }) = H_i \subset \GL_{N_i}(\F_q)$ with by assumption
$H_i \supset \SL_{N_i}(\F_q)$, and $H^i \vartriangleleft H_i$. If the
$H_i / H^i$ are both abelian, then we can conclude by the argument
used in Lemma~\ref{lem:levi}. Otherwise, they are both non-abelian,
hence $H^i \subset \F_q^{\times}$. The only non-abelian simple composition
factor of $H_i /H^i$ being $\PSL_{N_i}(\F_q)$, this is possible
only if $N_1 = N_2 = N/2$. Let then $\overline{R_i} : B_{n } \to \PGL_{N_i}(\F_q)$
the projective representation of $B_{n }$ deduced from $R_i$. By the
very description of $H$ we have $\overline{R_2}(b) = \hat{\varphi}( R_1(b))$ for
all $b \in B_{n }$, where $\hat{\varphi}$ is the composite
$H_1 \onto H_1/H^1 \stackrel{\simeq}{\to} H_2/H^2 \to \PGL_{N_2}(\F_q)$.
Since $\varphi$ is an isomorphism and $Z(H_i/H^i)$ is the image of $\F_q^{\times}
\cap H_i$ inside $H_i/H^i$, we have $\hat{\varphi}(\F_q^{\times} \cap H_1) = 1$,
hence $\overline{R_2}(b) = \check{\varphi} (\overline{R_1}(b))$ for all
$b \in B_{n }$, where $\check{\varphi} : H_1 / (\F_q^{\times} \cap H_1) \to \PGL_{N_2}(\F_q)$
is the induced morphism. Note that $H_1 / (\F_q^{\times} \cap H_1) \subset \PGL_{N_1}(\F_q)$,
and clearly $\mathrm{Im} \check{\varphi} \supset \PSL_{N_2}(\F_q)$. From this
one deduce that the restriction of $\check{\varphi}$
to $\PSL_{N_1}(\F_q)$ is non-trivial, hence induces an isomorphism $\psi$ between the simple
groups $\psi : \PSL_{N_1}(\F_q) \to \PSL_{N_2}(\F_q)$. Up to a possible
conjugation of the representations $R_1,R_2$, we get (see \cite{sol}
Theorem~2.5.12) that
$\psi$ is either induced by a field automorphism $\Phi \in \Aut(\F_q)$,
or by the composition of such an automorphism with $X \mapsto \ ^t X^{-1}$.
In the first case we let $S = R_1$, in the second case we let $S : g \mapsto ^t R_1(g^{-1})$.
In both cases, we have $\overline{R_2}(b) = \Phi(\overline{S}(b)) = \overline{ S^{\Phi}}(b)$ for all
$b \in B'_{n }$, with $S^{\Phi} : g \mapsto \Phi(S(g))$, meaning that the two representations of $B_{n }'$ afforded
by $R_2$ and $S^{\Phi}$ are projectively equivalent, that is there is $z : B'_{n } \to \F_q^{\times}$
such that $R_2(b) = S^{\Phi}(b) z(b)$ for all $b \in B'_{n }$. Since $B'_{n }$ is perfect
for $n  \geq 5$ (see \cite{GORINLIN}) we get $z = 1$ ; this proves that the restrictions of $R_2$ and $S^{\Phi}$
to $B'_{n }$ are isomorphic. In particular, their restrictions to $B'_3$ are isomorphic.
The restrictions of $R_2$ and $S$ to $B'_3$ are direct sums of the irreducible
representations of $TL_3$, restricted to the derived subgroups. There
are two such irreducible representations, of dimensions 1 and 2,
corresponding to the partitions $[3]$ and $[2,1]$, respectively. 
Note that these restrictions have to contain a component
of dimension $2$, for otherwise the image of $B'_3$
would be trivial, hence $\sigma_1$ and $\sigma_2$
would have the same image (as $\sigma_1 \sigma_2^{-1} \in B'_3$), which easily implies that the image of $B_{n }$
is abelian, contradicting either $R_2(B_{n }) \supset \SL_{N_2}(\F_q)$
or $R_1(B_{n }) \supset \SL_{N_1}(\F_q)$. But this implies that
the representation of $B'_3$ associated to $[2,1]$ has to be isomorphic
to its twisted by $\Phi$. By explicit computation we get that
the trace of $\sigma_1 \sigma_2 \sigma_1^{-1} \sigma_2^{-1}$
is $1 - (\alpha + \alpha^{-1})$. Since $\F_q = \F_p(\alpha) = \F_p(\alpha + \alpha^{-1})$
this implies $\Phi = 1$.

We thus have $R_2(b) = S(b)$ for all $b \in B'_{n-1}$. By Lemma~\ref{lem:gprime}
and because the abelianization of $B_n$ is given by $\ell : B_n \onto \Z$,
$\sigma_i \mapsto 1$, this means that $R_2(b) = S(b)u^{\ell(b)}$
for some $u \in \F_q^{\times}$, and this for all $b \in B_{n-1}$.

If $S = R_1$,
this implies that the spectrum of $R_2(\sigma_1)$, which is
made of $-1$ with multiplicity $a(n,r_2)$ 
and $\alpha$ with multiplicity $b(n,r_2)$, is
also made of  $-u$ with multiplicity $a(n,r_1)$ 
and $u \alpha$ with multiplicity $b(n,r_1)$.
Since $a(n,r_1) > b(n,r_1)$ and $a(n,r_2) > b(n,r_2)$ by
Lemma~\ref{lem:multiTL}, this implies $u=1$, hence
$R_2 = R_1$, which is excluded because these
two representations of the Temperley-Lieb algebras
are non-isomorphic by assumption.

Finally, if $S(b) = \ ^ t R_1(b^{-1})$ for all $b \in B_{n}$,
the spectrum of $R_2(\sigma_1)$, 
made of $-1$ with multiplicity $a(n,r_2)$ 
and $\alpha$ with multiplicity $b(n,r_2)$, is
also made of  $-u$ with multiplicity $a(n,r_1)$ 
and $u \alpha^{-1}$ with multiplicity $b(n,r_1)$.
Again by Lemma~\ref{lem:multiTL},
this implies $u = 1$, and $\alpha = \alpha^{-1}$
hence $\alpha^2 = 1$, contradicting the assumption on
the order of $\alpha$. This concludes the proof.



\end{proof}

Corollary~\ref{corR1R2SL} is an immediate consequence of Proposition~\ref{propR1R2SL}.

\subsection{The case of $[n-1,1]$}
\label{par32}

We first prove part (i) of Theorem~\ref{main} (under the assumption that
Theorem~\ref{theorinductGLGL} holds) when $\lambda=[n-1,1]$.
This case is essentially dealt by Theorem~\ref{theowagner},
as we show now.

Indeed, $R(B_n)$ is generated by the $R(\sigma_i)$, which
in our case $\la = [n-1,1]$ have eigenvalues $\alpha$ with multiplicity
$1$ and $-1$ with multiplicity $n-2$, and thus the $(-1)R(\sigma_i)$ are semisimple
pseudo-reflections of order at least 4 as $\alpha$ has order not dividing $6$.


When $R(B_n)$ is primitive and $n \geq 4$,
Theorem~\ref{theowagner} states that $\SL_N(\F_{\tilde{q}}) \leq R(B)
\leq \GL_N(\F_{\tilde{q}})$ for some $\tilde{q}$ dividing $q$, 
or $R(B) \leq \GU_N(\tilde{q})$ for some $\tilde{q}^2 | q$. Notice
now that $\det \sigma_1 = \alpha$. If $R(B) \subset \GU_N(\tilde{q})$,
$\alpha$ would be fixed by $\Gal(\F_{\tilde{q}^2}| \F_{\tilde{q}})$
which embeds in $\Gal(\F_q | \F_p)$, and this would contradict
$\F_p(\alpha) = \F_q$. Then $\SL_N(\F_{\tilde{q}}) \leq R(B)
\leq \GL_N(\F_{\tilde{q}})$, and $\F_q = \F_p(\alpha) \subset \F_{\tilde{q}}$
implies $\tilde{q} = q$ and the conclusion. We thus only need to prove the primitiveness,
and to take separate care for the case $n=3$. In this case, one matrix model is given by :

$$
\sigma_1 \mapsto s_1 = \begin{pmatrix}
-1 & -1 \\ 0 & \alpha
\end{pmatrix}\quad
\sigma_2 \mapsto s_2 = \begin{pmatrix}
-1 & 0 \\ 1+\alpha+ \alpha^2 & \alpha
\end{pmatrix}.
$$
We let $\bar{s}_1$ and $\bar{s}_2$ their image
in $\PGL_2(\F_q)$. We prove the following lemma
\begin{lemma} \label{lem:cas21} If the order of $\alpha$ is not in
$\{1,2,3,4,5,6,10 \}$,
then $\langle s_1,
s_2 \rangle$ contains $\SL_2(\F_q)$.
\end{lemma}
\begin{proof}
Let $G = \langle s_1, s_2 \rangle \subset \GL_2(\F_q)$, and choose
$u \in \overline{\F_p}$ such that $u^2 = -\alpha^{-1}$. We let
$\F_{q'} = \F_q(u)$ and $\overline{G}$ the image of $G$
in $\PGL_2(\F_q) \subset \PGL_2(\F_{q'})$. Then $\overline{G} = 
\langle \bar{s}_1,\bar{s}_2 \rangle 
= 
\langle \bar{us}_1,\bar{us}_2 \rangle \subset \PSL_2(\F_{q'})$.
By Dickson's theorem (see \cite{HUPPERT} Chapter II, Theorem~8.27),
we know that $\overline{G}$ is either abelian by abelian,
or isomorphic to one of the groups $\mathfrak{A}_5$,  $\mathfrak{S}_4$,
$\PSL_2(\F_{\tilde{q}'})$ 
or $\PGL_2(\F_{\tilde{q}'})$ for $\tilde{q}' | q'$ and $\tilde{q}'\geq
4$.

We first prove that $\overline{G} \subset \PSL_2(\F_{q'})$ cannot be
abelian by abelian. For this we note that
$s_1 s_2^{-1}$ and
$s_1^{-1} s_2$ belong to the image of $(B_3,B_3)$, hence 
the commutator subgroup of $[\overline{G},\overline{G}]$ contains
the commutator of $\overline{s}_1 \overline{s}_2^{-1}$
and $\overline{s}_1^{-1} \overline{s}_2$, which is non trivial
because 

$$
(s_1 s_2^{-1})(s_1^{-1} s_2)-(s_1^{-1} s_2)(s_1 s_2^{-1}) = 
\begin{pmatrix}
- \frac{- \alpha^2 + \alpha^3 -1}{\alpha^2} & - \frac{(\alpha -1)(1+\alpha^2)}{\alpha^2} \\
\frac{(\alpha-1)(\alpha^2+\alpha+1)}{\alpha} & \frac{\alpha^3 + \alpha -1}{\alpha} 
\end{pmatrix}
$$
is non-scalar when $\alpha$ has order at least $4$.

Now note that $\overline{s}_1^r = 1$ means that $u^r = u^{-r}$,
that is $\alpha^r = (-1)^r$. Our conditions thus imply that
$r \not\in \{1, 2,3,4,5,6 \}$, which rules out the cases
$\overline{G} \simeq \sym_4$ and $\overline{G} \simeq \mathfrak{A}_5$.
This proves that $[\overline{G},\overline{G}] \simeq
\PSL_2(\F_{\tilde{q}'})$
for some $\tilde{q}' | q'$. Since $[\overline{G},\overline{G}]
= \overline{[G,G]} \subset \PSL_2(\F_q)$ this implies
$\tilde{q}' | q$, so we denote $\tilde{q}' = \tilde{q} | q$.

The natural embedding $\overline{G}\subset \PSL_2(\F_{q'})$ can be
considered as a projective representation $\overline{\rho}$ of
$[\overline{G},\overline{G}]\simeq\PSL_2(\F_{\tilde{q}})$ with
associated cocyle $c\in Z^2(\PSL_2(\F_{\tilde{q}}),\F_{q'}^{\times})$. When
$\SL_2(\F_{\tilde{q}})$ is the Schur cover of $\PSL_2(\F_{\tilde{q}})$,
then $c$ becomes cohomologous to $0$ inside
$Z^2(\SL_2(\F_{\tilde{q}}),\F_{q'}^{\times})$ by the universal
coefficient theorem and because $\PSL_2(\F_{\tilde{q}})$ and
$\SL_2(\F_{\tilde{q}})$ are perfect. 
In the two cases where this does not hold, that is $\PSL_2(\F_4)$ and 
$\PSL_2(\F_9)$, we check on the Brauer character tables 
that every $2$-dimentional irreducible projective
representations in natural characteristic of these groups can be
linearized when lifted to $\SL_2(\F_4)$ and $\SL_2(\F_9)$, respectively. 
Moreover, $\SL_2(\F_{\tilde{q}})$ admits
%
a unique non-trivial representation in natural characteristic, up to
twisting by a field automorphism. This implies that
$[\overline{G},\overline{G}]$ is \emph{conjugated} to
$\Phi(\PSL_2(\F_{\tilde{q}}))$ for some $\Phi \in \Aut(\F_q)$ over
$\overline{\F}_p$.
Now, the trace of $s_1 s_2^{-1} \in \SL_2(\F_q)$ 
belongs
to $\Phi(\F_{\tilde{q}})$, hence $\F_q = \F_p(\alpha+\alpha^{-1}) \subset \Phi(\F_{\tilde{q}})$
since
$$
s_1 s_2^{-1} = \begin{pmatrix}
- (\alpha + \alpha^{-1}) & - \alpha^{-1} \\
\alpha^2 + \alpha+1 & 1
\end{pmatrix}.
$$
This proves $\tilde{q} = q$ and the conclusion by Lemma~\ref{lem:quotientsln}.

\end{proof}


Now we can get the conclusion for $[n-1,1]$ by induction on $n$, provided $n\geq 4$ :
since $R(B_{n-1}) \subset \GL_{n-1}(\F_q)$ has been shown to contain
$\SL_{n-2}(\F_q)$ by the branching rule and the induction assumption, $R$ is primitive,
hence contains $\SL_{n-1}(\F_q)$ by the previous argument. This concludes
the case $[n-1,1]$, and in particular the cases $n \leq 3$.

\subsection{Induction step for Theorem \ref{main}}
\label{par33}

We now can prove part (i) of Theorem \ref{main} (under the assumption
that Theorem \ref{theorinductGLGL} holds) by induction on $n$, and
restrict to the partitions $\la$ which are not of the form $[n]$ or
$[n-1,1]$, as the former case is trivial and the latter has been dealt
with in \S\ref{par32} (in particular this settles the initial cases $n \leq 3$).

When $n = 4$,
the additional $\la$ is $[2,2]$, which is an immediate
consequence of the case $\la = [2,1]$; indeed, one is deduced
from the other through the `special morphism'
$B_4 \to B_3$ which maps $s_3,s_1 \mapsto s_1$ and $s_2 \mapsto s_2$.
When $n=5$, the only case to consider is the $5$-dimensional
representation $\la = [3,2]$, for which
the restriction to $B_{n-1}$ is the direct sum of $[4,1]$ (3-dimensional)
and $[2,2]$ (2-dimensional). By Lemma~\ref{lem:levi} and the case $n=4$
we get that $R(B_4) \supset \SL_3(\F_q) \times \SL_2(\F_q)$ and
we get $R(B_5) \supset \SL_5(\F_q)$ by Theorem~\ref{theorinductGLGL}.
Finally, when $n \geq 6$, we can use Corollary~\ref{corR1R2SL} to
get the result by Theorem~\ref{theorinductGLGL}, except for the case $n = 2m$, $\la = [m,m]$,
in which case $c(2m,m) = c(2m-1,m-1)$ and one immediately gets $R(B_n) \supset R(B_{n-1}) \supset
\SL_{c(2m,m)}(\F_q)$ from the induction assumption,
and when $\la = [n-r,r]$ with $c(n-1,r) = c(n-1,r-1)$. Letting
$N = \dim \la$ we get in this case $R(B_n) \supset \SL_{N/2}(\F_q) \times
\SL_{N/2}(\F_q)$, and Theorem~\ref{theorinductGLGL}
implies $R(B_n) \supset \SL_N(\F_q)$ or $R(B_n) \subset \GL_{N/2}(\F_q)
\wr \mathfrak{S}_2$. We need to exclude the latter case.
For this, note that the composite $B_n \to \GL_{N/2}(\F_q)
\wr \mathfrak{S}_2 \to \mathfrak{S}_2$ factorizes through
$B_n^{ab}$, hence $R(B_n') \subset \GL_{N/2}(\F_q) \times \GL_{N/2}(\F_q)$.
Since $R(B_{n-1}) \subset \GL_{N/2}(\F_q) \times \GL_{N/2}(\F_q)$
and because $B_n$ is generated by $B_{n-1}$ and $B_n'$,
this implies $R(B_n) \subset \GL_{N/2}(\F_q) \times \GL_{N/2}(\F_q)$,
contradicting the irreducibility of $R$.
 This concludes the proof
of (i).

We now prove (ii). By (i), the image of $B_n$ inside each of the $\GL_{c(n,r)}(\F_q)$
contains $\SL_{c(n,r)}(\F_q)$, hence the image of $B_n'$ also contains $\SL_{c(n,r)}(\F_q)$.
We prove that the image of $B_n$ inside $\prod_{1 \leq r \leq k} \GL_{c(n,r)}(\F_q)$ 
contains $\prod_{1 \leq r \leq k} \SL_{c(n,r)}(\F_q)$ for $1 \leq k \leq n/2$
by induction on $k$. This amounts to saying that the image $H$ of $B_n'$
is $\prod_{1 \leq r \leq k} \SL_{c(n,r)}(\F_q)$. The cases $k = 1$ and
$k = 2$ are trivial, so we assume $k \geq 3$. We use Goursat's lemma with $G_1 = 
\prod_{1 \leq r \leq k-1} \SL_{c(n,r)}(\F_q)$ and $G_2 = \SL_{c(n,k)}(\F_q)$. By assumption and (i)
we have $H_1 = G_1$ and $H_2 = G_2$, and we get an isomorphism $ \varphi : H_1/H^1 \to H_2/ H^2$,
which induces a surjective morphism $\tilde{\varphi} : H_1 \to H_2/H^2$.

Assume that $H_1/H^1 \simeq H_2/H^2$ is not abelian. Then $H_2/H^2$ has for quotient $\PSL_{c(n,k)}(\F_q)$ and we get a surjective morphism
$\hat{\varphi} : H_1 \onto \PSL_{c(n,k)}(\F_q)$. Let now $r < k$, and consider
the restriction $\hat{\varphi}_r$ of $\hat{\varphi}$ to $\SL_{c(n,r)}(\F_q)$. Assume it
is non-trivial. Since the
image of the center is mapped to $1$, it factorizes through an isomorphism
$\check{\varphi}_r : \PSL_{c(n,r)}(\F_q) \to \PSL_{c(n,k)}(\F_q)$. But this implies that the image
of $B_n'$ inside $\SL_{c(n,r)}(\F_q) \times \SL_{c(n,k)}(\F_q)$ is included
inside $\{ (x,y) \ | \ \bar{y} = \check{\varphi}_r(\bar{x}) \}$, where $\bar{x}, \bar{y}$
denote the canonical images of $x,y$. On the other hand, we know by
Proposition~\ref{propR1R2SL}
that the image is all $\SL_{c(n,r)}(\F_q) \times \SL_{c(n,k)}(\F_q)$, a contradiction
that proves that each $\hat{\varphi}_r$ is trivial, hence so is $\hat{\varphi}$. Since it
is surjective, this provides a contradiction which excludes this case.
Thus $H_1/H^1 \simeq H_2/H^2$ is abelian, and we can conclude as in the proof of Lemma~\ref{lem:levi}.

\begin{remark} Note that we cannot immediately apply Lemma~\ref{lem:levi}, in order to prove
part (ii) of theorem \ref{main},
because it may happen that $c(n,r+1)=  c(n,r)$, for instance
$c(7,3) = c(7,4) = 14$.
\end{remark}

\subsection{Unitary case}
\label{par34}

We now assume $\F_p(\alpha + \alpha^{-1}) \neq \F_p(\alpha) = \F_q$,
and denote $\eps \in \Aut(\F_q)$ the generator of $\Gal(\F_q | \F_{q^{\frac{1}{2}}})$.
We first prove that, in the semisimple case and over a finite field, all representations
of the Hecke algebra are unitary.

\begin{prop} If  $\F_p(\alpha + \alpha^{-1}) \neq \F_p(\alpha) = \F_q$,
and $R : B_n \to \GL_N(\F_q)$ is an absolutely irreducible
representation associated to a partition $\la \vdash n$,
then there exists $P \in \GL_N(\F_q)$ such that $P R(B_n) P^{-1} \subset
\GU_N(q^{1/2})$.
\end{prop}
\begin{proof}
First note that, in that case, $\F_p(\alpha + \alpha^{-1}) 
= \F_{q^{1/2}}$, and
let $\eps$ be the generator (of order 2) of
$\Gal(\F_{q}|\F_{q^{1/2}})$. Then $\eps$
exchanges the roots $\alpha$ and $\alpha^{-1}$ of
the polynomial $X^2 - (\alpha + \alpha^{-1}) X + 1$. According to 
Lemma~\ref{lemrepunit} we only need to prove that $\eps \circ \rho^* \simeq
\rho$. Recall that two irreducible representations of the Hecke algebra
$H_{n}(\alpha)$ for $n \geq 4$
are isomorphic if and only if their restriction to
the Hecke algebra of type $H_{{n-1}}(\alpha)$ are isomorphic :
this means that two Young diagrams of size $n \geq 3$ are the same if
and only if the set of all their subdiagrams of size $n-1$ are the same,
and this is a simple exercice in the combinatorics of Young diagrams.
Since the restriction of representations commutes with twisting
by $\eps$ and taking the dual, this
readily proves the statement $\eps \circ \rho^* \simeq
\rho$ by induction on $n$, provided we know how to prove
it for $n = 2$. In that case however, it is trivial because all
irreducible representations are 1-dimensional, and given
by $\sigma_1 \mapsto -1$, $\sigma_1 \mapsto \alpha $.
Unitarity in that case simply means $\alpha^{-1} = \eps(\alpha)$, and
this concludes the proof.
\end{proof}

Up to conjugating the representations, we can thus assume $R(B_n) \subset
\GU_N(q^{1/2})$. One can then mimic the proof of the part (i) of the
theorem for the case $\F_p(\alpha) = \F_p(\alpha+\alpha^{-1})$.
Indeed,
\begin{itemize}
\item we have a similar statement as Lemma~\ref{lem:gprime} for
representations $R_1,R_2 : G \to \GU_N(q^{1/2})$ for $N \geq 2$ with
$(R_1)_{|G'} = (R_2)_{|G'}$ for $N \geq 2$ and $R_1(G') = R_2(G') \supset
\SU_N(q^{1/2})$ : then there exists $\eta : G \to \F_{q}^{\times}$
such that $R_2 = R_1 \otimes \eta$. Indeed, the same proof applies, because
the centralizer of $\SU_N(q^{1/2})$ in $\GL_N(\F_{q})$ is $(\F_{q})^{\times}$.
\item when $n = 3$, the same argument and Dickson's theorem imply (with the
same notations as in the proof of Lemma \ref{lem:cas21}) that $[\overline{G},\overline{G}] = 
\PSL_2(\F_{q^{1/2}}) = \PSU_2(q^{1/2})$ hence $[G,G] \supset \SU_2(q^{1/2})$.
\item we have a statement similar to Proposition~\ref{propR1R2SL} for
the unitary case namely that, with
the notations of this proposition, if $\F_p(\alpha + \alpha^{-1}) \neq
\F_p(\alpha) = \F_q$, and $R_i(B_n) \supset \SU_{N_i}(q^{1/2})$,
then $R(B_n) \supset \SU_{N_1}(q^{1/2}) \times \SU_{N_2}(q^{1/2})$.
The proof is similar, additional care being
needed only when considering the possible automorphisms
$\psi$ of $\PSU_{N/2}(q^{1/2})$. Up to possible (unitary) conjugation
of $R_1$ and $R_2$, $\psi$ is again either induced by
$\Phi \in \Aut(\F_{q})$ or by the composition by
such a field automorphism with $\mathfrak{s} : X \mapsto ^t \eps(X)^{-1}$
(see \cite{sol} Theorem 2.5.12).
Here $q = p^{2f}$, and $\Phi = F^r$ for $F : x \mapsto x^p$
the Frobenius automorphism and some $0 \leq r<2f$, and we can assume
$r < f$ because the actions of $F^f$ and $\mathfrak{s}$ coincide
on $\PSU_{N/2}(q^{1/2})$. We need to prove $\Phi = 1$,
and for this we are similarly reduced to considering
the case $[2,1]$.Then the final condition that $1- (\alpha + \alpha^{-1})$
is fixed implies that $\Phi \in \Gal(\F_{q}|\F_{q^{1/2}}) = \{ 1, F^f \}$
hence $\Phi = 1$ since $r<f$. The conclusion is then
similar, using the analogue of Lemma~\ref{lem:gprime}.

\item One uses Theorem~\ref{theorinductGUGU} instead of 
Theorem~\ref{theorinductGLGL}.

\end{itemize}

\section{Proofs of Theorem~\ref{theorinductGLGL} and of
Theorem~\ref{theorinductGUGU}}

For any finite group $H$ and any prime $p$, we denote by $O_p(H)$ the
unique maximal normal $p$-subgroup of $H$.
We will need the following result, that we derive from Kantor~\cite{kantor}, Theorem II (see also \cite{SZ1}).
\begin{theor} 
Let $p$ be a prime and $q$ be a $p$-power.
Suppose that $H$ is an irreducible subgroup of $\SL_N(\F_q)$ generated by a conjugacy
class of transvections, such that $O_p(H)\leq [H,H]\cap
\operatorname{Z}(H)$. Then $H$ is one of the following
groups.
\begin{enumerate}
\item $H=\SL_N(\F_{q'})$ or $H=\operatorname{Sp}_N(\F_{q'})$ in
$\SL_N(\F_{q'})$, or $H=\operatorname{SU}_N(q'^{1/2})$ in
$\SL_N(\F_{q'})$
\item $H=\operatorname{O}^{\pm}_N(\F_{q'})<\SL_N(\F_{q'})$, $q'$ even.
\item $H=\mathfrak{S}_{n}<\SL_{N}(\F_2)$, where $N=n-d$ with
$d=\operatorname{gcd}(n,2)$.
\item $H=\mathfrak{S}_{2n}$ in $\SL_{2n-1}(\F_2)$ 
or in $\SL_{2n}(\F_2)$.
\item $H=3\cdot \mathfrak{A}_6<\SL_3(\F_4)$.
\item $H=\SL_2(\F_5)<\SL_2(\F_9)$.
\item $H=3\cdot
\operatorname{P\Omega}^{-,\pi}_6(\F_3)<\SL_6(\F_4)$.
\item $H=\operatorname{SU}_4(2)<\SL_5(\F_4)$.
\item $H=A\rtimes\mathfrak{S}_N$ in $\SL_{N}(\F_{2^i})$ where $A$ is a subgroup of diagonal matrices.
\end{enumerate}
where $q'|q$.
\label{theo:kantor}
\end{theor}

Assume that $N\geq 5$ and $\pr\geq N-\pr$. In particular, $\pr\geq 3$.
Let $\Gamma$ be an absolutely irreducible subgroup of $\GL_N(\F_q)$
containing $\SL_{\pr}(\F_q)$. Write $G=[\Gamma,\Gamma]$. Note that
$\SL_{\pr}(\F_q)\leq G$ (because
$\SL_{\pr}(\F_q)$ is perfect since $\pr>2$).
Denote by $V=\F_q^N$ the natural representation of $\SL_N(\F_q)$. Let
$t$ be a transvection in $\SL_{\pr}(\F_q)$. Write $G_0=G$ and for every
$i\geq 1$, define $G_i$ the subgroup of $G_{i-1}$ generated by the
conjugacy class of $t$ in $G_{i-1}$. Note that $G_i$ is a normal
subgroup of $G_{i-1}$ and that $\SL_{\pr}(\F_q)\leq G_i$ (because
$\SL_{\pr}(\F_q)\leq G$ is generated by the conjugacy class of $t$ in
$\SL_{\pr}(\F_q)$).

First, assume that there is a positive integer $i$ such that $V$ is
an irreducible $\F_qG_j$-module for every $0\leq j <i$ and as an
$\F_qG_i$-module, $V$ is reducible. Note that if such an $i$ exists,
then $i>0$ because $V$ is an irreducible $\F_qG_0$-module by assumption.
Since $G_i$ is normal in $G_{i-1}$ and $V$ is an
irreducible $\F_qG_{i-1}$-module, Clifford's theorem (see for
example~\cite{CR}, \S11A) implies that 
$$\Res_{G_i}^{G_{i-1}}(V)=\bigoplus_{k=1}^r W_k,$$
where $W_k$ are irreducible $\F_qG_i$-modules and the $W_k$ are
$G_{i-1}$-conjugate to $W_1$.
Moreover, we can choose $W_1$ to be an irreducible component of
$\Res_{G_i}^{G_{i-1}}(V)$ such that the natural representation $V_{\pr}$
of $\SL_{\pr}(\F_q)$ is a component of
$\Res_{\SL_{\pr}(\F_q)}^{G_i}(W_1)$.  Hence, $$\dim(W_k)=\dim(W_1)\geq
\dim(V_{\pr})=a,$$  and we deduce that $r=2$ and $a=N/2$. In particular,
$G_i$ is a subgroup of $\GL(W_1)\times\GL(W_2)\leq \GL_N(\F_q)$, and
since $W_1$ and $W_2$ are $G_{i-1}$-conjugate, there is $g\in G_{i-1}$
such that ${}^g\GL(W_1)=\GL(W_2)$. Note that as vector space, we have
$W_1=V_{\pr}$. Thus, $\SL(W_1)\leq G_i$. However,
${}^g\SL(W_1)$ is a normal subgroup of $\GL(W_2)$ isomorphic to
$\SL_{\pr}(\F_q)$, hence ${}^g\SL(W_1)=\SL(W_2)$. Since $G_i$ is normal
in $G_{i-1}$, we obtain
$$\SL(W_1)\times\SL(W_2)\leq G_i.$$
Now, using that $G_i\leq \GL(W_1)\times\GL(W_2)$, we deduce that
$[G_i,G_i]=\SL(W_1)\times\SL(W_2)$. In particular,
$\SL(W_1)\times\SL(W_2)$ is a characteristic subgroup of $G_i$. Thus,
$G_{i-1}$ normalizes $\SL(W_1)\times\SL(W_2)$. However,
$\operatorname{N}_{\GL_N(\F_q)}(\SL(W_1)\times\SL(W_2)) =
\GL_{N/2}(\F_q)\wr \Z/2\Z$ and we get
$$G_{i-1}\leq \GL_{N/2}(\F_q)\wr \Z/2\Z.$$
Now, we prove by induction on  $0\leq j\leq i-1$ that $G_j\leq
\GL_{N/2}(\F_q)\wr \Z/2\Z$. We have shown that this is true for
$G_{i-1}$. Assume it holds for $G_j$. Then we have $\SL(W_1)\times\SL(W_2)
\leq G_j\leq \GL_{N/2}(\F_q)\wr\Z/2\Z$. Thus, $\SL(W_1)\times\SL(W_2)$ is
characteristic in $G_j$ (because the second derived subgroup of $G_j$ is
$\SL(W_1)\times\SL(W_2)$).
Since $G_j$ is normal in $G_{j-1}$, we deduce that $G_{j-1}$ normalizes
$\SL(W_1)\times\SL(W_2)$ and we conclude as above that $G_{j-1}\leq
\GL_{N/2}(\F_q)\wr\Z/2\Z$. In particular, $G=G_0$ is a subgroup of
$\GL_{N/2}(\F_q)\wr\Z/2\Z$. Now, since $G$ is the derived subgroup of
$\Gamma$, we deduce that the third derived subgroup of $\Gamma$ is
$\SL(W_1)\times\SL(W_2)$. Thus,
$\SL(W_1)\times\SL(W_2)$ is a characteristic subgroup of $\Gamma$ and
we conclude with the same argument that $\Gamma\leq
\GL_N(\F_q)\wr\Z/2\Z$.

Now, we assume that $V$ is an irreducible $\F_qG_j$-module for all
non-negative integer $j$. Note that there is a positive integer $r$ such
that $G_{r}=G_{r+1}$ (because the groups $G_i$ are finite). In
particular, $G_r$ is generated by the class of $t$ in $G_r$ and $V$ is
an irreducible $\F_qG_r$-module. By Clifford theorem,
$\Res_{O_p(G_r)}^{G_r}(V)$ is semisimple. However, as $p$-group, the
only irreducible $\F_qO_p(G_r)$-module of $O_p(G_r)$ is the trivial
module. Thus, $\Res_{O_p(G_r)}^{G_r}(V)$ is trivial, which implies that
$O_p(G_r)=1$. So, we can apply Theorem~\ref{theo:kantor} to $G_r$.
Note that $\SL_{\pr}(\F_q)\leq G_r$ implies that $q'=q$.

If $G_r=\SU_N(q^{1/2})$ or $G_r=\Sp_N(\F_q)$ or
$G_r=\operatorname{O}_N^{\pm}(\F_q)$ (for $q$ even), then the
contragredient representation $\rho^*$ of the natural representation
$\rho:G_r\rightarrow \GL_N(\F_q)$ would satisfy either
$\rho^*\simeq\rho$ or $\rho^*\simeq\varepsilon\circ\rho$, where
$\varepsilon \in\operatorname{Aut}(\F_q)$ has order $2$. Since its
restriction to $\SL_{\pr}(\F_q)\subset G_r$ does not (because $\pr\geq
3$), this is a contradiction.

Note that $G_r$ contains matrices whose coefficients
are not in $\F_2$ (because $q>3$). Hence, the cases (iii) and (iv) are excluded.
The cases (v) and (vi) are excluded, because $N\geq 5$.
Recall from \cite{kantor} that case (vii) actually corresponds to
an embedding of $G_r$ into $\operatorname{SU}_6(2)$.
This excludes the possibility that $G_r$ contains $\SL_3(\F_4)$,
because every 3-dimensional subspace of $\F_4^6$ contains
a nonzero vector $v$ of norm $0$, and therefore $G_r$
would contain all the transvections $x \mapsto x + \la \langle v | x \rangle v$
for $\la \in \F_4$. But such a transvection is an isometry
only if $\la + \bar{\la} = 0$, and therefore any $\la \in \F_4 \setminus \F_2$
provides a contradiction and thus case (vii) is excluded.

Now, note that $7$ does not divide the order of the group
$\operatorname{SU}_4(\F_2)$. But $7$ divides $|\SL_{3}(\F_4)|$, and a
fortiori $|\SL_{\pr}(\F_4)|$
excluding the cases (vii) and (viii). 

Now, suppose that $G_r=A\rtimes\mathfrak S_N$, where $A$ is a subgroup
of diagonal matrices and $\mathfrak S_N$ is identified with the subgroup
of permutation matrices of $\SL_N(\F_q)$.
Let $g\in G_r$. Then $g$ is a transvection of $G_r$ if and only if $g$
has order $2$ and has only one Jordan block $J_2(1)=\begin{pmatrix}
1&1\\ 0&1 \end{pmatrix}$ in its Jordan decomposition. Write
$\theta:G_r\rightarrow \mathfrak{S}_N$ for the natural projection. Note
that $g$ is conjugate to a block-diagonal matrix, whose block-matrices
$A_1,\ldots,A_k$ correspond to the decomposition of $\theta(g)=c_1\cdots
c_k$ into cycles with disjoint support.  Furthermore, if $c_i$ has
length $l$, then $A_i$ is a $l\times l$-matrix of order greater than
$l$. Then the $c_i$'s are transpositions and
$A_i=\begin{pmatrix}
0&a\\
a^{-1}&0
\end{pmatrix}$ for some $a\in\F_q^{\times}$, because $A_i^2=1$.
It follows that the characteristic polynomial of $A_i$ is $(X-1)^2$, and
$A_i$ is conjugate in $\GL_2(\F_q)$ to $J_2(1)$, because $A_i$ is
non-trivial.  Hence, the Jordan decomposition of $g$ consists in $k$
Jordan blocks $J_2(1)$. Therefore, if $g$ a transvection, then $k=1$
and $\theta(g)$ is a transposition.  Conversely, the matrix $t(a,i,j)=(t_{kl})_{1\leq k,l\leq n}$ , for $a\in\F_q^{\times}$
and $1\leq i<j\leq N$, defined by
$t_{ii}=t_{jj}=0$, $t_{kk}=1$ for $k\neq i,j$, $t_{ij}=a$,
$t_{ji}=a^{-1}$ and $t_{kl}=0$ otherwise, is a
transvection of $\SL_{N}(\F_q)$.
This proves that the number of transvections in $G_r$ is at most
$$T=(q-1)\frac{N(N-1)}{2}.$$
Moreover, recall that the transvections of $\SL_k(\F_q)$ are the set of
linear transformations $$t_{\varphi,v}:\F_q^k\rightarrow\F_q^k,\ x\mapsto
x+\varphi(x)v,$$ where
$\varphi$ is a non-zero linear form and $v\in\ker(\varphi)$ is a
non-zero vector.
Moreover, $t_{\varphi,v}=t_{\varphi',v'}$ if and only if there is a
scalar $\alpha\in\F_q^{\times}$ such that $\varphi'=\alpha\varphi$ and
$v=\alpha v'$. The number of transvections in $\SL_k(\F_q)$ is then
$$T'=\frac{(q^k-1)(q^{k-1}-1)}{q-1}.$$
Put $k=\lfloor N/2\rfloor$ and define $f$ by
$f(x)=(1+x+\cdots+x^{k-1})(1+x+\cdots +x^{k-2})-k(2k-1)$. Note
that if $f(q)> 0$, then $T'> T$. Suppose $N\geq
6$. Then $k\geq 3$. Moreover, we have $f(q)\geq f(4)$,
because $f$ is increasing. Using the fact that $4^i=(3+1)^i\geq 1+3i$
for $i\geq 1$, we obtain
$$f(4)\geq (1+3(k-1))(1+3(k-2))-k(2k-1)=7k^2-20k+10> 0.$$ 
Assume now that $N=5$. Then $k=2$ and $f(q)=1+q-4>0$ for $q\geq 4$.
This proves that $T'>T$, excluding $G_r=A\rtimes\mathfrak S_N$.
Finally, $G_r=\SL_N(\F_q)$, and $\SL_N(\F_q)\leq \Gamma$, as
required.\\

We prove Theorem~\ref{theorinductGUGU} in the same way. First, recall
that if $(k,q)\notin\{(2,2),\,(2,3),\,(3,3)\}$, then $\SU_k(q)$ is
perfect and $\PSU_k(q)$ is simple. Note that in this case, if $H$ is a
subgroup of $\SU_k(q)$ generated by a non-central conjugacy class of
$\SU_k(q)$, then $H=\SU_k(q)$. Indeed, write
$\pi:\SU_k(q)\rightarrow\PSU_k(q)$ for the natural projection. Then
$\pi(H)$ is a non-trivial normal subgroup of $\PSU_k(q)$. It follows
that $\pi(H)=\PSU_k(q)$ (because $\PSU_k(q)$ is simple). Hence,
$\SU_k(q)=HZ$, where $Z=\ker(\pi)$ is the center of $\SU_k(q)$.
Moreover, $$[H,H]=[HZ,HZ]=[\SU_k(q),\SU_k(q)]=\SU_k(q),$$ because
$\SU_k(q)$ is perfect, and the result follows.

So, assume that $N\geq 5$, $\pr\geq N-\pr$ and $q>3$. In particular,
$\SU_{\pr}(q)$ is perfect. Let $\Gamma$ be a subgroup of $\SU_N(q)$
containing $\SU_{\pr}(q)$, and write $G=[\Gamma,\Gamma]$. Then
$\SU_{\pr}(q)\leq G$. Let $t$ be a root element of $\SU_{\pr}(q)$, that
is a generator of a root subgroup.  Put $G_0=G$ and for every $i\geq 1$,
denote by $G_{i}$ the subgroup of $G_{i-1}$ generated by the conjugacy
class of $t$ in $G_{i-1}$. Since $t$ is not central in $\SU_{\pr}(q)$,
it follows from the above discussion that $\SU_{\pr}(q)\leq G_i$ for all
$i\geq 0$.

Now, the same argument as the one for $\SL_N$ gives that if the natural
representation $V$ of $\SU_N(q)$ is not $G_j$-irreducible for some
$j>0$, then $\Gamma\leq\GU_{N/2}(q)\wr \Z/2\Z$, and otherwise, there is
some positive integer $r$ such that $V$ is an irreducible
$\F_{q^2}G_r$-module and $G_r$ is generated by the conjugacy class of
$t$ in $G_r$. Thanks to~\cite{kantor},\,\S11, our assumptions, and the
fact that $G_r$ contains matrices with coefficients lying inside $\F_{q^2}$
and in no proper subfield of $\F_{q^2}$, we conclude that $G_r$ is either
$\SU_N(q)$, or $\operatorname{Sp}_N(q^2)$, or
$\operatorname{O}^{\pm}_N(q^2)$ (for $N$ and $q$ even), or  $3\cdot
\operatorname{P\Omega}^{-,\pi}(6,3)$ in $\SU_6(q)$, or $A\rtimes
\mathfrak S_N$ in $\SU_N(q)$, $q$ even, $A=a^{n-1}$ and $a|q+1$. The
cases $G_r= \operatorname{Sp}_N(q^2)$ and
$G_r=\operatorname{O}^{\pm}_N(q^2)$ are excluded, again because the
natural representation of $\SU_{\pr}(\F_q)$ is not self-dual for
$\pr\geq 3$.

Furthermore, the case
$G_r=3\cdot \operatorname{P\Omega}^{-,\pi}(6,3)$ is excluded,
 because $13$ divides $|\SU_3(4)|$ and does not
divide $|3\cdot \operatorname{P\Omega}^{-,\pi}(6,3)|$.

Assume now that $G_r=A\rtimes \mathfrak S_N$ with $A$ a subgroup of the
diagonal matrices of $\SU_N(q)$ of order $a^{N-1}$ with $a|q+1$. Then
the same argument as above shows that the number of transvections in
$G_r$ is at most $$T=(q+1)\frac{N(N-1)}{2}.$$
Note that a transvection $t_{\varphi,v}$ of $\SL_k(\F_q)$ is unitary if
its adjoint is equal to its inverse $t^{-1}_{\varphi,v}=t_{-\varphi,v}$.
This means that $v$ determines $\varphi$ (up to a scalar) and that $v$
is isotopic. The number of unitary transvections is then the number of
non-zero isotropic vectors divided by
$|\{\lambda\in\F_{q^2}^{\times}\,|\,\lambda^{q+1}=1\}|=q+1$. By
induction on $k$, we get that the number $a_{k-1}+(
(q^2)^{k-1}-a_{k-1}-1)(q+1)$. Hence,
$a_k=q^{2k-1}+(-q)^{k}+(-q)^{k-1}-1$, and the number of unitary
transvections is
$$T'=\frac{( (-q)^k-1)( (-q)^{k-1}-1)}{(-q)-1}.$$
(Note that conformally to the principle of Ennola's duality, this is the
same formula as before by replacing $q$ with $-q$.)

Now, for $k\geq 1$, let $h_k(x)=\frac{( (-x)^k-1)(
(-x)^{k-1}-1)}{-(q+1)^2}$. If $k=\lfloor N/2\rfloor$ then
$h_k(q)-k(2k-1)>0$ implies that $T'>T$, and we get the conclusion.
 Since for $m\geq 1$, the functions $x\mapsto
\frac{x^m-1}{x+1}$ and $x\mapsto \frac{x^m+1}{x+1}$ are increasing and
positive for $x\geq 1$, the same holds for $h_k$ when $k \geq 3$. Suppose
that $k$ is even and $k\geq 4$. Then
\begin{eqnarray*}
h_k(q)-k(2k-1)&=&\frac{(q^k)(q^{k-1}+1}{(q+1)^2}-k(2k-1)\\
&\geq&
\frac{(4^k-1)(4^{k-1}+1)}{25}-k(2k-1)\geq
\frac{(4^{k-1}-1)(4^{k-1}+1)}{25}-k(2k-1)
\end{eqnarray*}
is positive as soon as $4^{2(k-1)}-1-25k(2k-1)>0$. This is easily
checked to hold true. Suppose $k$ is odd and $k\geq 3$. Once again, we
have
$$h_k(q)-k(2k-1)=\frac{(q^k+1)(q^{k-1}-1)}{(q+1)^2}-k(2k-1)>0$$
if $k\geq 5$.
Moreover, $$h_3(q)-15=\frac{(q^3+1)(q^3-1)}{(q+1)^2}-15 >0$$when $q\geq
4$. It remains to consider when $\lfloor N/2\rfloor=2$, meaning in our
case $N=5$. Then $T=10(q+1)$ and the number of unitary transvections in
$\SU_3(q)$ is $$\frac{(q^3+1)(q^2-1)}{(q+1)^2}>T$$ when $q\geq 3$.
This proves the result.

\end{document}